\documentclass[11pt, a4paper]{article}
\usepackage{}
\usepackage{amsthm}
\usepackage{mathrsfs}
\usepackage{amsmath,amssymb,latexsym,color}
\usepackage[colorlinks,
linkcolor=blue,
anchorcolor=green,
citecolor=magenta
]{hyperref}
\usepackage{graphicx}
\usepackage{tikz}
\usetikzlibrary{calc}
\usepackage{CJK}

\usepackage{authblk}

\long\def\delete#1{}

\usepackage{color}

\definecolor{Blue}{rgb}{0,0,1}
\definecolor{Red}{rgb}{1,0,0}
\definecolor{DarkGreen}{rgb}{0,0.6,0}
\definecolor{DarkYellow}{rgb}{1,1,0.2}
\definecolor{DarkPurple}{rgb}{.6,0,1}

\usepackage{xcolor}
\usepackage[normalem]{ulem}

\usepackage{cleveref}
\crefformat{section}{\S#2#1#3}
\crefformat{subsection}{\S#2#1#3}
\crefformat{subsubsection}{\S#2#1#3}
\crefrangeformat{section}{\S\S#3#1#4 to~#5#2#6}
\crefmultiformat{section}{\S\S#2#1#3}{ and~#2#1#3}{, #2#1#3}{ and~#2#1#3}

\def\ma{\mathcal{A}}
\def\mb{\mathcal{B}}
\def\mc{\mathcal{C}}

\def\mf{\mathcal{F}}
\def\mg{\mathcal{G}}
\def\mh{\mathcal{H}}

\def\ml{\mathcal{L}}
\def\mm{\mathcal{M}}
\def\mn{\mathcal{N}}

\def\mr{\mathcal{R}}

\def\mt{\mathcal{T}}

\def\bs{\setminus}

\def\ge{\geqslant}
\def\le{\leqslant}

\def\ro{\romannumeral}

\newtheorem{thm}{Theorem}[section]
\newtheorem{lem}[thm]{Lemma}

\newtheorem{ex}{Construction}

\newtheorem{cl}{Claim}

\newtheorem{as1}{Assumption}

\begin{document}
	\setcounter{page}{1}
	\renewcommand{\thefootnote}{}
	\newcommand{\remark}{\vspace{2ex}\noindent{\bf Remark.\quad}}
	\renewcommand{\abovewithdelims}[2]{%
		\genfrac{[}{]}{0pt}{}{#1}{#2}}

	
	\def\qed{\hfill$\Box$\vspace{11pt}}
	
\title{Large non-trivial $t$-intersecting families for signed sets}
\author{Tian Yao\thanks{E-mail: \texttt{yaotian@mail.bnu.edu.cn}}}
\author{Benjian Lv\thanks{E-mail: \texttt{bjlv@bnu.edu.cn}}}
\author{Kaishun Wang\thanks{Corresponding author. E-mail: \texttt{wangks@bnu.edu.cn}}}
\affil{Laboratory of Mathematics and Complex Systems (Ministry of Education), School of
	Mathematical Sciences, Beijing Normal University, Beijing 100875, China}
\date{}

\maketitle
\begin{abstract}
	For positive integers $n,r,k$ with $n\ge r$ and $k\ge2$, a set $\{(x_1,y_1),(x_2,y_2),\dots,(x_r,y_r)\}$ is called a $k$-signed $r$-set on $[n]$ if $x_1,\dots,x_r$ are distinct elements of $[n]$ and $y_1\dots,y_r\in[k]$. 
	We say a $t$-intersecting family consisting of $k$-signed $r$-sets on $[n]$ is trivial if each member of this family contains a fixed $k$-signed $t$-set. In this paper, we determine the structure of large maximal non-trivial $t$-intersecting families. In particular, we characterize the non-trivial $t$-intersecting families with maximum size for $t\ge2$, extending a Hilton-Milner-type result for signed sets given by Borg.
	
			\vspace{2mm}
	
	
	\
	
\end{abstract}
\section{Introduction}

Let $n$, $r$ and $t$ be positive integers with $n\ge r\ge t$. For an $n$-set $X$, let $2^X$ and $\binom{X}{r}$ denote the collection of all subsets and the set of all $r$-subsets of $X$, respectively. 
A family $\mf\subset2^X$ is called $t$-\emph{intersecting} if $|F\cap F'|\ge t$ for any $F,F'\in\mf$. 
Moreover, we say $\mf$ is \emph{trivial} if each member of $\mf$ contains a fixed $t$-subset of $X$. 

The famous Erd\H{o}s-Ko-Rado Theorem \cite{EKR,Fn,Wn} states that the largest $t$-intersecting subfamilies of $\binom{X}{r}$ are trivial if $n>(t+1)(r-t+1)$. 
In \cite{Fn}, Frankl conjectured the structure of the maximum sized $t$-intersecting subfamilies of $\binom{X}{r}$ for all $n,r$ and $t$.  
Frankl's conjecture was partially settled by Frankl and F\"{u}redi \cite{FF}, and was completely confirmed by Ahlswede and Khachatrian \cite{SEKRT}.

The maximum sized non-trivial $t$-intersecting subfamilies of $\binom{X}{r}$ have been characterized. Hilton and Milner \cite{SHM} gave the first result for the structure of such families when $t=1$, which was also proved by Frankl and F\"{u}redi \cite{FFF} via the shifting technique. 
In \cite{SHMT1}, Frankl proved the corresponding result  for all $t$ and sufficiently large $n$. 
The complete result was given by Ahlswede and Khachatrian \cite{SHMT}. 
Besides, some other maximal non-trivial $t$-intersecting subfamilies of $\binom{X}{r}$ have been characterized. 
Han and  Kohayakawa \cite{LSET3} described the structure of the second largest maximal non-trivial $1$-intersecting familes with $n>2r\ge6$. 
Kostochka and Mubayi \cite{LSET1} determined the structure of $1$-intersecting families with sizes quit a bit smaller than $\binom{n-1}{r-1}$ for large $n$. 
Recently, Cao et al. \cite{LSET2} gave the structure of large maximal non-trivial $t$-intersecting families for all $t$ and large $n$.

The $t$-intersection problem has been studied on some other mathematical objects, for example, signed sets. 
Write $[n]=\{1,2,\dots,n\}$. For $k\ge2$,  
each element of 
$$\ml_{n,r,k}:=\left\{\{(x_1,y_1),\dots,(x_r,y_r)\}:\{x_1,\dots,x_r\}\in\binom{[n]}{r},y_1,\dots,y_r\in[k]\right\}$$
is called a $k$-\emph{signed} $r$-\emph{set on} $[n]$. 
For two subfamilies $\mf$ and $\mg$ of $\ml_{n,r,k}$, if there exists a bijection $\sigma$ from $[n]\times[k]$ to itself such that $\mg=\{\sigma(F): F\in\mf\}$, then we say $\mf$ is \emph{isomorphic} to $\mg$, and denote this by $\mf\cong\mg$.

The Erd\H{o}s-Ko-Rado Theorem for signed sets states that if $n$ or $k$ is sufficiently large, then the maximum sized $t$-intersecting subfamilies of $\ml_{n,r,k}$ are trivial, see \cite{HEKR1,HEKR2,HEKR3} for $r=n$ and \cite{SSET2,SSET1} for $r<n$.

In this paper, we study the structure of maximal non-trivial $t$-intersecting subfamilies of $\ml_{n,r,k}$. 
To present our main results, we introduce two constructions of non-trivial $t$-intersecting subfamilies of $\ml_{n,r,k}$. For each $d\in[n]$, write $M_d=\{(1,1),(2,1),\dots,(d,1)\}$.

\begin{ex}
	Suppose that $n,r,k,\ell$ and $t$ are positive integers with $2\le k$, $t+1\le r\le n$ and $t+2\le \ell\le\min\{r+1,n\}$. Let $\mh_1(n,r,k,\ell,t)$ be the set of all elements $F$ of $\ml_{n,r,k}$ such that
	\begin{itemize}
		\item $M_t\subset F$ and $|F\cap M_\ell|\ge t+1$, or
		\item $M_t\not\subset F$ and $|F\cap M_\ell|=\ell-1$.
	\end{itemize}
\end{ex}
\begin{ex}
	Suppose that $n,r,k,c$ and $t$ are positive integers with $2\le k$, $t+2\le r\le n$ and $r+2\le c\le\min\{2r-t,n\}$. Let $\mh_2(n,r,k,c,t)$  be the set of all elements $F$ of $\ml_{n,r,k}$ such that
	\begin{itemize}
		\item $M_t\subset F$ and $|F\cap M_r|\ge t+1$, or
		\item $F\cap M_r=M_t$ and $M_c\bs M_r\subset F$, or
		\item $M_t\not\subset F$, $|F\cap M_r|=r-1$ and $|F\cap(M_c\bs M_r)|=1$.
	\end{itemize}
\end{ex}

For each $d\in[n]$, write
\begin{equation}\label{f}
	f(n,r,k,d,t)=(d-t)\binom{n-t-1}{r-t-1}k^{r-t-1}-\binom{d-t}{2}\binom{n-t-2}{r-t-2}k^{r-t-2},
\end{equation}
\begin{equation}\label{g}
	g(n,r,t)=\dfrac{(r-t+3)(r-t-1)}{n-t-1}\cdot\max\left\{\binom{t+2}{2},\dfrac{r-t+1}{2}\right\}.
\end{equation}
One of our main results describes the structure of maximal non-trivial $t$-intersecting subfamilies of $\ml_{n,r,k}$ with sizes no less than $f(n,r,k,r,t)$.

\begin{thm}\label{T1}
	Let $n,r,k$ and $t$ be positive integers with  $n\ge t+2$, $n\ge r\ge t+1$ and $k\ge\max\{2,g(n,r,t)\}$. 
	Suppose that $\mf$ is a maximal non-trivial $t$-intersecting subfamily of $\ml_{n,r,k}$.  
	Then $|\mf|\ge f(n,r,k,r,t)$ if and only if one of the following holds.
	\begin{itemize}
		\item[\rm{(\ro1)}] $r\ge t+2$ and $\mf\cong\mh_1(n,r,k,m,t)$ for some $m\in\{r,\min\{r+1,n\}\}$.
		\item[\rm{(\ro2)}] $n\ge r+2\ge t+4$ and $\mf\cong\mh_2(n,r,k,c,t)$ for some $c\in\{r+2,\dots,\min\{2r-t,n\}\}$.
		\item[\rm{(\ro3)}] $r\le2t+2$, $r\neq t+2$ and $\mf\cong\mh_1(n,r,k,t+2,t)$.
	\end{itemize}
\end{thm}

In \cite{SSETHM1}, Borg determined the structure of the largest non-trivial $1$-intersecting subfamilies of $\ml_{n,r,k}$. Our second main result extends Borg's result.

\begin{thm}\label{T3}
	Let $n,r,k$ and $t$ be positive integers with $n\ge t+2\ge4$, $n\ge r\ge t+1$ and $k\ge\max\{2,g(n,r,t)\}$. 
	Suppose that $\mf$ is a largest non-trivial $t$-intersecting subfamily of $\ml_{n,r,k}$. 
	\begin{itemize}
		\item[\rm{(\ro1)}] If $\min\{r+1,n\}\le2t+2$, then $\mf\cong\mh_1(n,r,k,t+2,t)$.
		\item[\rm{(\ro2)}] If $\min\{r+1,n\}>2t+2$, then $\mf\cong\mh_1(n,r,k,\min\{r+1,n\},t)$.
	\end{itemize}
\end{thm}

The rest of this paper is organized as follows.
In Section \ref{S2}, we will prove some properties for $t$-intersecting families with $t$-covering number $t+1$ in preparation for the proof of our main results. In Section \ref{S3} and \ref{S4}, we will prove Theorems \ref{T1} and \ref{T3}, respectively.

\section{$t$-intersecting families with $t$-covering number $t+1$}\label{S2}

For a $t$-intersecting subfamily $\mf$ of $\ml_{n,r,k}$, 
a $k$-signed set $T$ on $[n]$ is said to be a $t$-\emph{cover} of $\mf$ if $|T\cap F|\ge t$ for each $F\in\mf$, and the minimum size of $\mf$'s $t$-covers is called the $t$-\emph{covering number} $\tau_t(\mf)$ of $\mf$. 
Observe that $t\le\tau_t(\mf)\le r$, and $\mf$ is trivial if and only if $\tau_t(\mf)=t$. 
In this section, we prove some properties for $t$-intersecting subfamilies of $\ml_{n,r,k}$ with $t$-covering number $t+1$. 

For convenience, we write $\mf_X:=\{F\in\mf: X\subset F\}$ where $\mf$ is a subset of $\ml_{n,r,k}$ and $X$ a $k$-signed set on $[n]$. We also make the following assumption.

\begin{as1} \label{ass1}
	Let $n,r,k$ and $t$ be positive integers with $n\ge r\ge t+1$ and $k\ge2$. Suppose $\mf\subset\ml_{n,r,k}$ is a maximal $t$-intersecting family with $\tau_t(\mf)=t+1$. Let $\mt$ denote the set of all $t$-covers of $\mf$ with size $t+1$. Set $M=\bigcup\limits_{T\in\mt}T$ and $\ell=|M|$.
\end{as1}

 We first claim that $\mt$ is a $t$-intersecting family with $t\le\tau_t(\mt)\le t+1$. In fact, for $T\in\mt$ and $F\in\ml_{n,r,k}$ containing $T$, we have $F\in\mf$ by the maximality of $\mf$. Then for each $T'\in\mt$, there exists $F'\in\mf$ such that $T'\subset F'$ and $T'\cap T=F'\cap F$, which implies that $|T'\cap T|\ge t$, as desired. 
 To describe the structure of some $t$-intersecting families, we need the following lemma.
\begin{lem}\label{XJ}
	Let $n,r,k,t,\ell,\mf,\mt$ and $M$ be as in Assumption \ref{ass1}. 
	\begin{itemize}
		\item[\rm{(\ro1)}] If $\tau_t(\mt)=t+1$, then $M\in\ml_{n,t+2,k}$ and $|F\cap M|\ge t+1$ for each $F\in\mf$.
		
		\item[\rm{(\ro2)}] If $\tau_t(\mt)=t$, then $M\in\ml_{n,\ell,k}$ with $t+1\le\ell\le\min\{r+1,n\}$, and $|F\cap M|=\ell-1$ for each $F\in\mf\bs\mf_S$,  where $S$ is a $t$-cover of $\mt$ with size $t$.
	\end{itemize}
\end{lem}
\begin{proof}
	(\ro1) Let $T_1$ and $T_2$ be distinct members of $\mt$. We claim that $T_1\Delta T_2\in\ml_{n,2,k}$. 
	Indeed, since $|T_1\cap T_2|=t$ and $\mf$ is non-trivially $t$-intersecting, we have $|T_1\Delta T_2|=2$ and there exists a member of $\mf\bs\mf_{T_1\cap T_2}$ containing $T_1\Delta T_2$, which imply that $T_1\Delta T_2\in\ml_{n,2,k}$.
	
	Since $\tau_t(\mt)=t+1$, there exists $T_3\in\mt$ such that $T_1\cap T_2\not\subset T_3$. 
	From $|T_1\cap T_3|\ge t$ and $|T_2\cap T_3|\ge t$, we get $T_1\Delta T_2\subset T_3$ and $|T_3\cap(T_1\cap T_2)|=t-1$, which imply that $T_3\subset T_1\cup T_2$. 
	For each $T_4\in\mt\bs\{T_1\}$ containing $T_1\cap T_2$, we have $T_1\cap T_3\not\subset T_4$.  	
	Similarly, we have $T_4\subset T_1\cup T_3\subset T_1\cup T_2$. 
	Hence $M\subset T_1\cup T_2\subset M$. 
	Then it follows from the claim that $M\in\ml_{n,t+2,k}$. 
	For each $F\in\mf$, we have $|F\cap M|\ge t$. If $|F\cap M|=t$, then $F\cap M$ is contained in each member of $\mt$, which is impossible. Therefore, $|F\cap M|\ge t+1$, as desired.

	(\ro2)  By the claim in (\ro1), it is routine to check that $M\in\ml_{n,\ell,k}$. 
	Let $S$ be a $t$-cover of $\mt$. 
	For each $F\in\mf\bs\mf_S$ and $T\in\mt$, we have $|F\cap T|=t$, from which we get $r+1\le|S\cup F|\le|T\cup F|=r+1$. 
	Then $S\cup F=T\cup F$, which implies that $|M\cup F|=|S\cup F|=r+1$. 
	Hence $|F\cap M|=\ell-1$ and $\ell\le r+1$. 
	Together with $M\in\ml_{n,\ell,k}$ and $\mt\neq\emptyset$, we obtain $t+1\le\ell\le\min\{r+1,n\}$, as required.
\end{proof}

For a $k$-signed set $Q=\{(s_1,t_1),\dots,(s_q,t_q)\}$ on $[n]$ with $s_1\le\dots\le s_q$, let $\pi_0=(q\ s_q)(q-1\ s_{q-1})\cdots(1\ s_1)\in S_n$, and for each $x\in[n]$,  let $\pi_x\in S_k$ with $\pi_x=(1\ t_i)$ if $x=s_i$ for some $i\in[q]$ and $\pi_x=(1)$ otherwise. 
We get a bijection $\pi$ from $[n]\times[k]$ to itself with $\pi(x,y)=(\pi_0(x),\pi_x(y))$ for each $(x,y)\in[n]\times[k]$.
Observe that $\pi(Q)=M_q$, and $\pi(\ml_{n,s,k})=\ml_{n,s,k}$ for each $s\in[n]$. 
It is routine to check that there exists a bijection $\sigma$ from $[n]\times[k]$ to itself such that $\sigma(\mf)$ is a $t$-intersecting subfamily of $\ml_{n,r,k}$ with $t$-covering number $t+1$,  $M_\ell=\bigcup_{T\in\mt'}T$, 
and $M_t$ is a $t$-cover of $\mt'$ if $\tau_t(\mt)=t$, where $\mt'$ is the set of all $t$-covers of $\sigma(\mf)$ with size $t+1$. 
Let $\mg$ denote the family $\sigma(\mf)$.

\begin{lem}\label{p1}
	Let $n,r,k,t,\ell,\mf,\mt$ and $M$ be as in Assumption \ref{ass1}. Suppose that $|F\cap M|\ge t+1$ for each $F\in\mf$.
	\begin{itemize}	
		\item[\rm{(\ro1)}] If $\tau_t(\mt)=t+1$, then $\mf\cong\mh_1(n,r,k,t+2,t)$.
		\item[\rm{(\ro2)}] If $\tau_t(\mt)=t$, then $\mf\cong\mh_1(n,r,k,\ell,t)$ where  $\ell\in\{t+3,\dots,\min\{r+1,n\}\}$.	
	\end{itemize}
\end{lem}
\begin{proof}
	
	(\ro1) \ If $\tau_t(\mt)=t+1$, by Lemma \ref{XJ} (\ro1) we have $M\in\ml_{n,t+2,k}$. By assumption and $\mf\cong\mg$, we have $|G\cap M_{t+2}|\ge t+1$ for each $G\in\mg$. Then $\mg\subset\mh_1(n,r,k,t+2,t)$. 
	Since $\mh_1(n,r,k,t+2,t)$ is $t$-intersecting and $\mg$ is maximal, we have $\mf\cong\mg=\mh_1(n,r,k,t+2,t)$.
	
	(\ro2) \  Since $\mf$ is non-trivially $t$-intersecting, by Lemma \ref{XJ} (\ro2), we have $t+2\le\ell\le\min\{r+1,n\}$. 
	Notice that each $(t+1)$-subset of $M_\ell$ containing $M_t$ is a $t$-cover of $\mg$. 
	Then $\{G\in\ml_{n,r,k}: M_t\subsetneq G\cap M_\ell\}\subset\mg$. By Lemma \ref{XJ} (\ro2), we have $|G\cap M_\ell|=\ell-1$ for each $G\in\mg\bs\mg_{M_t}$. Hence $\mg\subset\mh_1(n,r,k,\ell,t)$. Since $\mg$ is maximal and $\mh_1(n,r,k,\ell,t)$ is $t$-intersecting, we have $\mf\cong\mg=\mh_1(n,r,k,\ell,t)$. Notice that $\tau_t(\mt)=t+1$ if $\ell=t+2$. Then $\ell\ge t+3$, as desired.
\end{proof}

\begin{lem}\label{p2}
	Let $n,r,k,t,\ell,\mf,\mt$ and $M$ be as in Assumption \ref{ass1}. Suppose that there exists $F_0\in\mf$ such that $|F_0\cap M|=t$. Then $t\le r-2$ and $\ell<\min\{r+1,n\}$. Moreover, if $\ell=\min\{r+1,n\}-1$, then $r\le n-2$ and $\mf\cong\mh_2(n,r,k,c,t)$ for some $c\in\{r+2,\dots,\min\{2r-t,n\}\}$.
\end{lem}
\begin{proof}
	By Lemma \ref{XJ} (\ro1), we have $\tau_t(\mt)=t$. If $r=t+1$, then $\mt=\mf$, which implies that $\tau_t(\mt)=t+1$, a contraction. Hence $r\ge t+2$. 
	Observe that $F_0\cap M$ is a $t$-cover of $\mt$. 
	Let $F\in\mf\bs\mf_{F_0\cap M}$. 
	If $\ell=\min\{r+1,n\}$, by Lemma \ref{XJ} (\ro2), we have $|F\cap F_0|=|F\cap(F_0\cap M)|<t$, which is impossible. 
	Therefore, $\ell<\min\{r+1,n\}$.
	
	Now suppose that $\ell=\min\{r+1,n\}-1$. Since $\mf\cong\mg$, there exists $G_0\in\mg$ such that $|G_0\cap M_\ell|=t$. 
	 Let $G\in\mg\bs\mg_{M_t}$.
	If $r\ge n-1$, 
	then $\ell=n-1$. 
	By Lemma \ref{XJ} (\ro2), we have $|G_0\cap G\cap([n-1]\times[k])|=t-1$, which implies that $(n,x_0)\in G_0\cap G$ for some $x_0\in[k]$. 
	Then $M_t\cup\{(n,x_0)\}$ is a $t$-cover of $\mg$, which is impossible since $\ell<n$ and each member of $\mt'$ is contained in $M_\ell$. 
	Hence $r\le n-2$ and $\ell=r$.
	
	By $|G_0\cap G|\ge t$ and Lemma \ref{XJ} (\ro2), we obtain $G\bs([r]\times[k])\in\binom{G_0}{1}$. 
	Let $$E=\{(i,j):i\ge r+1,\ (i,j)\in G\ \mbox{for some}\ G\in\mg\bs\mg_{M_t}\}.$$
	Observe that $E$ is a non-empty subset of $G_0$ and $E\cap M_r=\emptyset$. 
	We have $1\le|E|\le\min\{r-t,n-t\}$. 
	If $E=\{(e_1,e_2)\}$ for some $e_1\ge t+1$ and $e_2\in[k]$, then $(e_1,e_2)$ is contained in each member of $\mg\bs\mg_{M_t}$, which implies that $M_t\cup\{(e_1,e_2)\}\in\mt'$, a contradiction. 
	Therefore $|E|\ge2$. 
	Since $M_t$ is a $t$-cover of $\mt'$, then each $(t+1)$-subset of $M_r$ containing $M_t$ is a member of $\mt'$, which implies that $\{G\in\ml_{n,r,k}: M_t\subsetneq G\cap M_r\}\subset \mg$. 
	For each $G_0'\in\mg_{M_t}$ with $|G_0'\cap M_r|=t$, observe that $G\bs([r]\times[k])\subset G_0'$. Then we have $E\subset G_0'$. 
	For each $G'\in\mg\bs\mg_{M_t}$, we have $|G'\cap M_r|=r-1$ and $G'\cap E\neq\emptyset$. 
	Together with $2\le|E|\le\min\{r-t,n-t\}$, it is routine to check that $\mg$ is isomorphic to a subset of $\mh_2(n,r,k,c,t)$ where $r+2\le c\le\min\{2r-t,n\}$. 
	Since that $\mg$ is maximal and $\mh_2(n,r,k,c,t)$ is $t$-intersecting, we have $\mf\cong\mg\cong\mh_2(n,r,k,c,t)$, as desired.
\end{proof}

Now we prove upper bounds for sizes of families under Assumption \ref{ass1} with $\tau_t(\mt)=t$. We begin with a frequently used lemma.
\begin{lem}\label{FS}
	Let $n,r,k,t$ and $u$ be positive integers with $n\ge r\ge u+1$. Suppose $\mf\subset\ml_{n,r,k}$ is a $t$-intersecting family and $U\in\ml_{n,u,k}$. If $|U\cap F|=s<t$ for some $F\in\mf$, then there exists $R\in\ml_{n,u+t-s,k}$ such that $|\mf_U|\le\binom{r-s}{t-s}|\mf_R|$.
\end{lem}
\begin{proof}
	W.l.o.g., assume that $\mf_U\neq\emptyset$. Let $\mr$ denote the set of $R\in\ml_{n,u+t-s,k}$ such that  $U\subset R\subset F\cup U$. 
	For $G\in\mf_U$, from $|G\cap F|\ge t$ and $|F\cap U|=s<t$, we obtain $|G\cap(F\cup U)|\ge u+t-s$, which implies that $\mr\neq\emptyset$ and $\mf_U=\bigcup_{R\in\mr}\mf_R$. Since $|F\cup U|=u+r-s$, we have $|\mr|\le\binom{r-s}{t-s}$. 
	Then the desired result holds by  $|\mf_U|\le\sum_{R\in\mr}|\mf_R|$.
\end{proof}

\begin{lem}\label{boundt+1}
	Let $n,r,k,t,\ell,\mf,\mt$ and $M$ be as in Assumption \ref{ass1} with $|\mt|=1$. Then
	\begin{equation*}\label{l=t+1}
		|\mf|\le\binom{n-t-1}{r-t-1}k^{r-t-1}+(t+1)(r-t)^2\binom{n-t-2}{r-t-2}k^{r-t-2}.
	\end{equation*}
\end{lem}
\begin{proof}
	Suppose that $T_0$ is the unique element of $\mt$. 	
	We have
	\begin{equation}\label{t+1,1}
		\mf=\mf_{T_0}\cup\left(\bigcup_{W\in\binom{T_0}{t}}\mf_{W}\bs\mf_{T_0}\right).
	\end{equation}
	For each $W\in\binom{T_0}{t}$, there exists $F_1\in\mf\bs\mf_{T_0}$ such that $|W\cap F_1|<t$. 
	Since $|F_1\cap T_0|=t$, we have $|F_1\cap W|=t-1$. 
	Let $H_1=F_1\cup W$. It is routine to check that $|H_1|=r+1$ and $T_0\subset H_1$. 
	For each $F_1'\in\mf_W\bs\mf_{T_0}$, we have $|F_1'\cap H_1|\ge t+1$ by $|F_1\cap F_1'|\ge t$. 
	Then
	\begin{equation}\label{t+1,2}
		\mf_{W}\bs\mf_{T_0}=\bigcup_{I\in\ml_{n,t+1,k}\bs\{T_0\},\ W\subset I\subset H_1}\mf_I\bs\mf_{T_0}.
	\end{equation}
	For each $I\in\ml_{n,t+1,k}\bs\{T_0\}$ with $W\subset I\subset H_1$, since $I\not\in\mt$, there exists $F_1''\in\mf$ such that $t-1\le|F_1''\cap W|\le|F_1''\cap I|\le t-1$. 
	Observe that $I\cup T_0\in\ml_{n,t+2,k}$. Since $\mf$ is maximal and $T_0$ is a $t$-cover of $\mf$, each element of $\ml_{n,r,k}$ containing $T_0$ is a member of $\mf$, which implies that $|\mf_{I\cup T_0}|=\binom{n-t-2}{r-t-2}k^{r-t-2}$. 
	By Lemma \ref{FS}, we have $|\mf_I|\le(r-t+1)\binom{n-t-2}{r-t-2}k^{r-t-2}$. 
	Then
	\begin{equation}\label{t+1,3}
		|\mf_I\bs\mf_{T_0}|=|\mf_I|-|\mf_{I\cup T_0}|\le(r-t)\binom{n-t-2}{r-t-2}k^{r-t-2}.
	\end{equation}
	Notice that  $|\mf_{T_0}|=\binom{n-t-1}{r-t-1}k^{r-t-1}$ and the number of $I\in\ml_{n,t+1,k}\bs\{T_0\}$ with $W\subset I\subset H_1$ is at most $r-t$. 
	Together with (\ref{t+1,1}), (\ref{t+1,2}) and (\ref{t+1,3}), we get the desired bound of $|\mf|$.
\end{proof}

\begin{lem}\label{bound>t+1}
	Let $n,r,k,t,\ell,\mf,\mt$ and $M$ be as in Assumption \ref{ass1} with $|\mt|\ge2$ and $\tau_t(\mt)=t$. 
	\begin{itemize}
		\item[\rm{(\ro1)}] If $\ell=t+2$, then
		$$|\mf|\le2\binom{n-t-1}{r-t-1}k^{r-t-1}+(r-1)(r-t+1)\binom{n-t-2}{r-t-2}k^{r-t-2}.$$
		\item[\rm{(\ro2)}] If $\ell\ge t+3$, then
		$$	|\mf|\le(\ell-t)\binom{n-t-1}{r-t-1}k^{r-t-1}+\left( (r-\ell+1)(r-t+1)+t\right) \binom{n-t-2}{r-t-2}k^{r-t-2}.$$
	\end{itemize}
\end{lem}
\begin{proof}
	Suppose that $S$ is a $t$-cover of $\mt$ with size $t$. 
	
	We first prove an upper bound for $|\mf_S|$. 
	Let $F_2\in\mf\bs\mf_S$ and $H_2=S\cup F_2$. 
	It follows from Lemma \ref{XJ} (\ro2) that  $M\subset H_2$ and $|H_2|=r+1$. 
	For each $F_2'\in\mf_S$, if $F_2\cap M=S$, from $|F_2\cap F_2'|\ge t$ we get $|F_2'\cap H_2|\ge t+1$. 
	Write
	$$\ma=\left\{A\in\ml_{n,t+1,k}: S\subset A\subset H_2, A\not\subset M\right\},\quad\mb=\left\{B\in\ml_{n,t+1,k}: S\subset B\subset M\right\}.$$
	Observe that each member of $\mf_S$ contains at least one element of $\ma\cup\mb$.
	For each $A\in\ma$, since $A\not\in\mt$, there exists $F_2''\in\mf$ such that $t-1\le|F_2''\cap S|\le|F_2''\cap A|\le t-1$. 
	Then by Lemma \ref{FS}, we have $|\mf_A|\le(r-t+1)\binom{n-t-2}{r-t-2}k^{r-t-2}$. 
	Notice that $|\ma|\le r-\ell+1$, $|\mb|=\ell-t$ and $|\mf_B|\le\binom{n-t-1}{r-t-1}k^{r-t-1}$ for each $B\in\mb$. 
	Then we obtain
	\begin{equation}\label{l,1}
		|\mf_S|\le(\ell-t)\binom{n-t-1}{r-t-1}k^{r-t-1}+(r-\ell+1)(r-t+1)\binom{n-t-2}{r-t-2}k^{r-t-2}.
	\end{equation}

	Let $\mc=\left\{C\in\ml_{n,\ell-1,k}: S\not\subset C\subset M\right\}$. We have $|\mc|=t$ and $\mf\bs\mf_S\subset\bigcup_{C\in\mc}\mf_C$.
	
	(\ro1) \ Suppose $\ell=t+2$. For each $C\in\mc$, since $C\not\in\mt$, there exists $F_3\in\mf$ such that $|F_3\cap C|\le t-1$. 
	Together with $|F_3\cap M|\ge t$, we have $|F_3\cap C|=t-1$. 
	By Lemmas \ref{XJ} (\ro2), \ref{FS} and $|\mc|=t$, we have
	\begin{equation*}\label{l,3}
		|\mf\bs\mf_S|\le\sum_{C\in\mc}|\mf_C|\le t(r-t+1)\binom{n-t-2}{r-t-2}k^{r-t-2}.
	\end{equation*}
	Together with (\ref{l,1}) we obtain the desired result.

	(\ro2) \ Suppose $\ell\ge t+3$. Observe that $|\mf_C|\le\binom{n-\ell+1}{r-\ell+1}k^{r-\ell+1}$  for each $C\in\mc$. By Lemma \ref{XJ} (\ro2), $\ell\ge t+3$ and $|\mc|=t$, we have
	\begin{equation*}\label{l,2}
		|\mf\bs\mf_S|\le\sum_{C\in\mc}|\mf_C|\le t\binom{n-\ell+1}{r-\ell+1}k^{r-\ell+1}\le t\binom{n-t-2}{r-t-2}k^{r-t-2}.
	\end{equation*}
	Together with (\ref{l,1}) we get the desired bound of $|\mf|$.
\end{proof}

\section{Proof of Theorem \ref{T1}}\label{S3}

Let $n,r,k$ and $t$ be positive integers with  $n\ge t+2$, $n\ge r\ge t+1$ and $k\ge\max\{2,g(n,r,t)\}$. 
Suppose that $\mf$ is a maximal non-trivial $t$-intersecting subfamily of $\ml_{n,r,k}$. 
If $r=t+1$, 
then $\tau_t(\mf)=t+1$ and $\mf$ is the set of its $t$-covers with size $t+1$. 
It follows from Lemmas \ref{XJ} (\ro1) and \ref{p1} (\ro1) that $\mf\cong\mh_1(n,t+1,k,t+2,t)$ and $|\mf|=t+2>1=f(n,t+1,k,t+1,t)$. 
In the following, we may assume that $r\ge t+2$. 
Write
$$\varphi(n,r,k,t)=\dfrac{f(n,r,k,r,t)-|\mf|}{\binom{n-t-2}{r-t-2}k^{r-t-2}}.$$
It is sufficient to show that $\varphi(n,r,k,t)<0$ if one of (\ro1), (\ro2) and (\ro3) in Theorem \ref{T1} holds, and $\varphi(n,r,k,t)>0$ otherwise.

	\medskip
	\noindent{\bf Case 1. $\tau_t(\mf)=t+1$.}
\medskip

In this case, let $\mt$ be the set of all $t$-covers of $\mf$ with size $t+1$ and $\ell=|\bigcup_{T\in\mt} T|$.  Note that $t\le\tau_t(\mt)\le t+1$, and $t+1\le\ell\le\min\{r+1,n\}$ by Lemma \ref{XJ}.

\medskip
\noindent{\bf Case 1.1. $\tau_t(\mt)=t$.} 
\medskip

In this case, (\ro3) does not hold. 

\medskip
\noindent{\bf Case 1.1.1. (\ro1) or (\ro2) holds.}
\medskip

 By Lemmas \ref{XJ} (\ro1), \ref{p1} and \ref{p2}, we have $\ell\ge r$. Let $a$ be an integer with $a\ge t+1$. For each $b\in\{t+1,\dots,a\}$, set 
$$\mn_b(M_a,M_t)=\{F\in\ml_{n,r,k}: M_t\subset F,|F\cap M_a|=b\}.$$
We claim that
\begin{equation}\label{slc}
	f(n,r,k,a,t)=\sum_{i=1}^{a-t}\dfrac{3i-i^2}{2}\cdot|\mn_{t+i}(M_a,M_t)|.
\end{equation}
For each $b\in\{t+1,\dots,a\}$, let $\mm_b(M_a,M_t)$ denote that set of all  $(I,F)\in\ml_{n,b,k}\times\ml_{n,r,k}$ with $M_t\subset I\subset M_a$ and $I\subset F$. 
By double counting $|\mm_{t+1}(M_a,M_t)|$ and $|\mm_{t+2}(M_a,M_t)|$, we obtain
\begin{equation*}\label{slc1}
	\sum_{i=1}^{a-t}i|\mn_{t+i}(M_a,M_t)|=(a-t)\binom{n-t-1}{r-t-1}k^{r-t-1},
\end{equation*}	
\begin{equation*}\label{slc2}
	\sum_{i=2}^{a-t}\binom{i}{2}|\mn_{t+i}(M_a,M_t)|=\binom{a-t}{2}\binom{n-t-2}{r-t-2}k^{r-t-2},
\end{equation*}
which imply that (\ref{slc}) holds. By (\ref{slc}) we have
\begin{equation}\label{>l}
f(n,r,k,\ell,t)\le|\mn_{t+1}(M,S)|+|\mn_{t+2}(M,S)|\le|\mf_S|<|\mf|.	
\end{equation}	
If $r=n$, we obtain $\varphi(n,r,k,t)<0$ from (\ref{>l}). 
If $r<n$, by (\ref{f}), (\ref{g}) and $k\ge g(n,r,t)$, we get
$$\dfrac{f(n,r,k,r+1,t)-f(n,r,k,r,t)}{\binom{n-t-2}{r-t-2}k^{r-t-2}}=\dfrac{(n-t-1)k}{r-t-1}-(r-t)>0.$$
Together with (\ref{>l}), we get $\varphi(n,r,k,t)<0$, as desired.

\medskip
\noindent{\bf Case 1.1.2. Neither (\ro1) nor (\ro2) holds.} 
\medskip

In this case, by Lemmas \ref{XJ} (\ro1), \ref{p1} and \ref{p2}, we have $\ell<r$. 
If $\ell=t+1$, from (\ref{f}), Lemma \ref{boundt+1} and $(n-t-1)k\ge\binom{t+2}{2}(r-t)^2$, we have
\begin{equation*}
	\begin{aligned}
		\varphi(n,r,k,t)&\ge(n-t-1)k-\binom{r-t}{2}-(t+1)(r-t)^2\ge\dfrac{(t^2+t-1)(r-t)^2}{2}>0.
	\end{aligned}
\end{equation*}
	If $\ell=t+2$, we have $r-t\ge3$. From (\ref{f}), (\ref{g}), Lemma \ref{bound>t+1} (\ro1) and $k\ge g(n,r,t)$, we obtain
\begin{equation*}
	\begin{aligned}
		\varphi(n,r,k,t)&\ge\dfrac{(r-t-2)(n-t-1)k}{r-t-1}-\binom{r-t}{2}-(r-1)(r-t+1)\\
		&\ge(r-t-2)(r-t+3)\left(\binom{t+2}{2}-\dfrac{3(r-t)^2+(2t-1)(r-t)+2(t-1)}{2(r-t-2)(r-t+3)}\right)\\ 
		&\ge(r-t-2)(r-t+3)\left( \binom{t+2}{2}-\dfrac{4t+11}{6}\right)\\
		&>0.
	\end{aligned}
\end{equation*}
If  $\ell\ge t+3$, we have $r-t\ge4$. Notice that
\begin{equation}\label{alpha}
\begin{aligned}
g(n,r,t)&\ge\left(\alpha\binom{t+2}{2}+\left(1-\alpha\right)\cdot\dfrac{r-t+1}{2}\right)\cdot\dfrac{(r-t+3)(r-t-1)}{n-t-1}\\
&\ge\left(t+\left(1-\dfrac{1}{3(r-t+3)}\right)\cdot\dfrac{(r-t+1)(r-t+3)}{2}\right)\cdot\dfrac{r-t-1}{n-t-1}\\
&=\left(t+\dfrac{3(r-t)^2+11(r-t)+8}{6}\right)\cdot\dfrac{r-t-1}{n-t-1},
\end{aligned}
\end{equation}
where $\alpha$ is a real number such that $\binom{t+2}{2}(r-t+3)\alpha=t$.
Together with (\ref{f}), (\ref{g}), Lemma \ref{bound>t+1} (\ro2), $k\ge g(n,r,t)$ and $r-\ell\ge1$, we get	
\begin{equation*}
	\begin{aligned}
		\varphi(n,r,k,t)
		&\ge\dfrac{(r-\ell)(n-t-1)k}{r-t-1}-\binom{r-t}{2}-\left(r-\ell+1\right)(r-t+1)-t\\
		&\ge(r-\ell)\left(\dfrac{(n-t-1)k}{r-t-1}-\binom{r-t}{2}-2(r-t+1)-t\right) \\
		&\ge\dfrac{3(r-t)^2+11(r-t)+8}{6}-\binom{r-t}{2}-2(r-t+1)\\
		&>0,
	\end{aligned}
\end{equation*}
as desired.

\medskip
\noindent{\bf Case 1.2. $\tau_t(\mt)=t+1$.} 
\medskip

In this case, by Lemmas \ref{XJ} (\ro1) and \ref{p1} 
(\ro1), we have $\mf\cong\mh_1(n,r,k,t+2,t)$. Then (\ro2) does not hold. Next we show that $\varphi(n,r,k,t)<0$ if (\ro1) or (\ro3) holds and $\varphi(n,r,k,t)>0$ otherwise. 
Observe that
	\begin{equation}\label{Z}
	|\mh_1(n,r,k,t+2,t)|=(t+2)\binom{n-t-1}{r-t-1}k^{r-t-1}-(t+1)\binom{n-t-2}{r-t-2}k^{r-t-2},
\end{equation}
and it follows from (\ref{f}) that
	\begin{equation}\label{Z0}
\varphi(n,r,k,t)=\dfrac{(r-2t-2)(n-t-1)k}{r-t-1}-\binom{r-t}{2}+(t+1).
\end{equation}

Firstly we suppose that (\ro1) or (\ro3) holds. 
Then $r\le 2t+2$. 
If $r=2t+2$, by (\ref{Z0}), we have
$$\varphi(n,r,k,t)=-\binom{t+2}{2}+(t+1)=-\binom{t+1}{2}<0.$$
If $r<2t+2$, by (\ref{g}), (\ref{Z0}) and $k\ge g(n,r,t)$, we get
$$\varphi(n,r,k,t)\le-\dfrac{(n-t-1)k}{r-t-1}-\binom{r-t}{2}+(t+1)\le-\binom{t+2}{2}(r-t+3)+(t+1)<0,$$
as desired.

Now suppose that (\ro3) does not hold. We have $r>2t+2$. From (\ref{g}), (\ref{Z0}) and $k\ge g(n,r,t)$, we obtain
$$		\varphi(n,r,k,t)\ge\dfrac{(n-t-1)k}{r-t-1}-\binom{r-t}{2}+(t+1)\ge\dfrac{(r-t+3)(r-t+1)}{2}-\binom{r-t}{2}>0,$$
as required.

\medskip
\noindent{\bf Case 2. $\tau_t(\mf)\ge t+2$.} 
\medskip

In this case, by Lemmas \ref{XJ} (\ro1), \ref{p1} and \ref{p2}, none of (\ro1), (\ro2) and (\ro3) holds. To prove $\varphi(n,r,k,t)>0$, we firstly prove an upper bound for $|\mf|$.
\medskip
\begin{cl}\label{cl1}
		$|\mf|\le(r-t+1)^2\binom{t+2}{2}\binom{n-t-2}{r-t-2}k^{r-t-2}.$
\end{cl}
\begin{proof}[Proof of Claim \ref{cl1}]
	Suppose $\tau_t(\mf)=z$ and $Z$ is a $t$-cover of $\mf$ with size $z$. 
For $Y_0\in\binom{Z}{t}$, w.l.o.g., assume that $\mf_{Y_0}\neq\emptyset$. 
Since $Y_0$ is not a $t$-cover of $\mf$, there exists $X_0\in\mf$ such that $|X_0\cap Y_0|<t$. 
By Lemma \ref{FS}, there exists $Y_1\in\ml_{n,2t-|Y_0\cap X_0|,k}$ containing $Y_0$ such that
$$|\mf_{Y_0}|\le\binom{r-|X_0\cap Y_0|}{t-|X_0\cap Y_0|}|\mf_{Y_1}|\le(r-t+1)^{t-|X_0\cap Y_0|}|\mf_{Y_1}|.$$
Note that $\mf_{Y_1}\neq\emptyset$ by $|\mf_{Y_0}|>0$. 
Similarly, we deduce that there exist $k$-signed sets $Y_0,Y_1,\dots,Y_w$ on $[n]$ such that $Y_0\subset\cdots\subset Y_w$ with $|Y_{w-1}|<z$, $|Y_w|\ge z$ and $$|\mf_{Y_i}|\le(r-t+1)^{|Y_{i+1}|-|Y_i|}|\mf_{Y_{i+1}}|$$
for each $i=0,\dots,w-1$. Therefore $$|\mf_{Y_0}|\le(r-t+1)^{|Y_w|-t}|\mf_{Y_w}|\le(r-t+1)^{|Y_w|-t}\binom{n-|Y_w|}{r-|Y_w|}k^{r-|Y_w|}.$$  
Together with $k\ge g(n,r,t)$,  we obtain
$$\dfrac{|\mf_{Y_0}|}{(r-t+1)^{z-t}\binom{n-z}{r-z}k^{r-z}}\le\prod_{i=z}^{|Y_w|-1}\dfrac{(r-t+1)(r-i)}{(n-i)k}\le\left(\dfrac{2}{r-t+3}\right)^{|Y_w|-z}\le1.$$
Notice that $\mf=\bigcup_{Y\in\binom{Z}{t}}\mf_Y$. 
Then
$$|\mf|\le(r-t+1)^{z-t}\binom{z}{t}\binom{n-z}{r-z}k^{r-z}.$$
For each $y\in\{t+2,\dots,r\}$, write $$\psi(y)=(r-t+1)^{y-t}\binom{y}{t}\binom{n-y}{r-y}k^{r-y}.$$ 
If $y\le r-1$, by $y\ge t+2$, $k\ge g(n,r,t)$ and (\ref{g}), we have
$$\dfrac{\psi(y+1)}{\psi(y)}
=\dfrac{y+1}{y+1-t}\cdot\dfrac{(r-t+1)(r-y)}{(n-y)k}\le\dfrac{t+3}{3}\cdot\dfrac{2(r-t+1)}{(t+1)(t+2)(r-t+3)}\le1.$$
Then from $z\ge t+2$, we get $|\mf|\le\psi(t+2)$, as desired.
\end{proof}
\medskip

Observe that 
\begin{equation*}
\begin{aligned}
g(n,r,t)&\ge\left((1-\beta)\binom{t+2}{2}+\beta\cdot\dfrac{r-t+1}{2}\right)\cdot\dfrac{(r-t+3)(r-t-1)}{n-t-1}\\
&=\left(\dfrac{(r-t)^2+3(r-t)+4}{r-t+1}\binom{t+2}{2}+\dfrac{1}{r-t}\binom{r-t}{2}\right)\cdot\dfrac{r-t-1}{n-t-1},
\end{aligned}
\end{equation*}
where $\beta$ is a real number such that $(r-t+3)(r-t+1)\beta=r-t-1$.
Together with (\ref{f}), (\ref{g}),  $r\ge t+2$, $k\ge g(n,r,t)$ and Claim \ref{cl1}, 
we have
\begin{equation*}
	\begin{aligned}
		\varphi(n,r,k,t)
		&\ge\dfrac{(r-t)(n-t-1)k}{r-t-1}-\binom{r-t}{2}-\binom{t+2}{2}(r-t+1)^2\\
		&\ge\binom{t+2}{2}\left(\dfrac{(r-t)^3+3(r-t)^2+4(r-t)}{r-t+1}-(r-t+1)^2\right)\\
		&=\dfrac{r-t-1}{r-t+1}\binom{t+2}{2}\\
		&>0.
	\end{aligned}
\end{equation*}	
This finishes the proof of Theorem \ref{T1}.\qed

\section{Proof of Theorem \ref{T3}}\label{S4}

Let $n,r,k$ and $t$ be positive integers with $n\ge t+2\ge 4$, $n\ge r\ge t+1$  and $k\ge\max\{2,g(n,r,t)\}$. 
Suppose that $\mf$ is a maximum sized non-trivial $t$-intersecting subfamily of $\ml_{n,r,k}$. If $r=t+1$, 
by Theorem \ref{T1}, we have $\mf\cong\mh_1(n,r,k,t+2,t)$. 
In the following, we assume that $r\ge t+2$. Write $p=\min\{r+1,n\}$. 
\medskip
\begin{cl}\label{cl2}
	$\mf$ is isomorphic to $\mh_1(n,r,k,p,t)$ or $\mh_1(n,r,k,t+2,t)$.
\end{cl}
\begin{proof}[Proof of Claim \ref{cl2}]
	
	Suppose for contradiction that neither $\mh_1(n,r,k,p,t)$ nor $\mh_1(n,r,k,t+2,t)$ is isomorphic to $\mf$.
Let $\mt$ be the set of all $t$-covers of $\mf$ with size $\tau_t(\mf)$ and $\ell=|\bigcup_{T\in\mt} T|$. By Theorem \ref{T1} and Lemmas \ref{XJ} (\ro1), \ref{p1}, \ref{p2}, we have $\tau_t(\mf)=t+1$, $\tau_t(\mt)=t$ and $\ell=r\neq p$. Therefore $n>r$, $p=r+1$ and $|\mt|\ge2$.

If $r=t+2$, by (\ref{f}), (\ref{g}), $k\ge g(n,r,t)$ and Lemma \ref{bound>t+1} (\ro1), we get
$$\dfrac{f(n,r,k,p,t)-|\mf|}{\binom{n-t-2}{r-t-2}k^{r-t-2}}\ge\dfrac{(n-t-1)k}{r-t-1}-\binom{r-t+1}{2}-3(r-1)\ge5\binom{t+2}{2}-3(t+2)>0.$$
If $r\ge t+3$, by (\ref{f}), (\ref{g}), (\ref{alpha}), $k\ge g(n,r,t)$ and Lemma \ref{bound>t+1} (\ro2), we  have
\begin{equation*}
	\begin{aligned}
		\dfrac{f(n,r,k,p,t)-|\mf|}{\binom{n-t-2}{r-t-2}k^{r-t-2}}
		&\ge\dfrac{(n-t-1)k}{r-t-1}-\binom{r-t+1}{2}-(r-t+1)-t\\
		&\ge\dfrac{3(r-t)^2+11(r-t)+8}{6}-\binom{r-t+1}{2}-(r-t+1)\\
		&>0.
	\end{aligned}
\end{equation*}
Together with (\ref{>l}), we get $|\mf|<f(n,r,k,p,t)\le|\mh_1(n,r,k,p,t)|$, a contradiction to the assumption that $\mf$ is maximum sized. 
\end{proof}
\medskip

If $n=t+2$, it follows from Claim \ref{cl2} that $\mf\cong\mh_1(n,r,k,t+2,t)$. In the following we may assume that $n\ge t+3$. 
Write $$\mu(n,r,k,t)=\dfrac{|\mh_1(n,r,k,t+2,t)|-|\mh_1(n,r,k,p,t)|}{\binom{n-t-2}{r-t-2}k^{r-t-2}}.$$ 
By Claim \ref{cl2}, it suffices to show that $\mu(n,r,k,t)<0$ if $p>2t+2$, and $\mu(n,r,k,t)>0$ if $p\le2t+2$. 
We divide the remaining proof into three cases.

\medskip
\noindent{\bf Case 1. $p>2t+2$.}
\medskip

Since $k\ge g(n,r,t)$ and $|\mh_1(n,r,k,p,t)|>f(n,r,k,p,t)$, by (\ref{f}), (\ref{g}) and (\ref{Z}), we have 
\begin{equation*}
	\begin{aligned}
		\mu(n,r,k,t)<-\dfrac{(n-t-1)k}{r-t-1}+\binom{p-t}{2}-(t+1)
		\le-\dfrac{3(r-t+1)}{2}-(t+1)
		<0,
			\end{aligned}
	\end{equation*}
as desired.

\medskip
\noindent{\bf Case 2. $p<2t+2$.}
\medskip

By the construction of $\mh_1(n,r,k,p,t)$, it is routine to verify that
\begin{equation*}
	|\mh_1(n,r,k,p,t)|\le(p-t)\binom{n-t-1}{r-t-1}k^{r-t-1}+t(k-1).
\end{equation*}
Then if $r\ge t+3$, by (\ref{g}), (\ref{Z}), $t\ge2$ and $k\ge g(n,r,t)$, we have
\begin{equation*}
	\begin{aligned}
		\mu(n,r,k,t)&\ge\dfrac{(n-t-1)k}{r-t-1}-(t+1)-t\ge\binom{t+2}{2}(r-t+3)-(2t+1)>0.
	\end{aligned}
\end{equation*}
If $r=t+2$, then $p=t+3$ by $n\ge t+3$, and
$$|\mh_1(n,t+2,k,t+3,t)|=3(n-t-1)k+t-3.$$
Together with (\ref{Z}), $n\ge t+3$ and $t,k\ge2$, we obtain
$$\mu(n,t+2,k,t)=(t-1)((n-t-1)k-2)>0,$$
as required.

\medskip
\noindent{\bf Case 3. $p=2t+2$.}
\medskip

In this case, we have $r\ge p-1>t+2$. 
By the construction of $\mh_1(n,r,k,p,t)$, we have
$$|\mh_1(n,r,k,p,t)|\le\sum_{i=1}^{p-t}|\mn_{t+i}(M_p,M_t)|+t(k-1).$$
Together with (\ref{slc}) and $|\mn_{t+i}(M_p,M_t)|\le\binom{t+2}{i}\binom{n-t-i}{r-t-i}k^{r-t-i}$ for each $i\in\{3,\dots,p-t\}$, we get
\begin{equation*}\label{h1-f}
\begin{aligned}
|\mh_1(n,r,k,p,t)|-f(n,r,k,p,t)&\le\sum_{i=3}^{p-t}\binom{i-1}{2}|\mn_{t+i}(M_p,M_t)|+t(k-1)\\
&\le\sum_{i=3}^{p-t}\binom{i-1}{2}\binom{t+2}{i}\binom{n-t-i}{r-t-i}k^{r-t-i}+t(k-1).
\end{aligned}	
\end{equation*}
For each $i\in\{3,\dots,p-t\}$, write
$$\lambda(i)=\binom{i-1}{2}\binom{t+2}{i}\binom{n-t-i}{r-t-i}k^{r-t-i}.$$
If $i\le p-t-1$, by (\ref{g}), $t\ge2$, $i\ge3$ and $k\ge g(n,r,t)$, we have 
\begin{equation*}
\begin{aligned}
\dfrac{\lambda(i+1)}{\lambda(i)}&=\dfrac{i(t+2-i)}{(i-2)(i+1)}\cdot\dfrac{r-t-i}{(n-t-i)k}\le\dfrac{3(t-1)}{4(t+1)(t+2)}\le\dfrac{1}{4}.
\end{aligned}
\end{equation*}
Then
\begin{equation*}
\begin{aligned}
|\mh_1(n,r,k,p,t)|-f(n,r,k,p,t)&\le\lambda(3)\cdot\sum_{j=0}^\infty\dfrac{1}{4^j}+t(k-1)\\
&=\dfrac{4}{3}\binom{t+2}{3}\binom{n-t-3}{r-t-3}k^{r-t-3}+t(k-1).
\end{aligned}
\end{equation*}
Together with (\ref{g}), $t\ge2$, $k\ge g(n,r,t)$ and
$$|\mh_1(n,r,k,t+2,t)|-f(n,r,k,p,t)=\binom{t+1}{2}\binom{n-t-2}{r-t-2}k^{r-t-2},$$ 
we get
\begin{equation*}\label{-1}
	\begin{aligned}
		\mu(n,r,k,t)&\ge\binom{t+1}{2}-t-\dfrac{4(r-t-2)}{3(n-t-2)k}\binom{t+2}{3}\\
		&\ge\binom{t}{2}-\dfrac{8}{3(t+1)(t+2)(r-t+3)}\cdot\dfrac{(t+2)(t+1)t}{6}\\
		&\ge\left(\dfrac{t-1}{2}-\dfrac{4}{9}\right)t\\
		&>0. 
	\end{aligned}
\end{equation*}
This finishes the proof of Theorem \ref{T3}.\qed

	\medskip
\noindent\textsc{Acknowledgement.} B. Lv is supported by NSFC (12071039, 11671043) and NSF of Hebei Province(A2019205092), K. Wang is supported by the National Key R\&D Program of China (No. 2020YFA0712900) and NSFC (12071039, 11671043).

\end{document}